\documentclass{article}
\usepackage{amsmath,amsthm,amsopn,amssymb}
\allowdisplaybreaks

\numberwithin{equation}{section}
\newtheorem{theorem}{Theorem}[section]
\newtheorem{lemma}{Lemma}[section]
\makeatother

\begin{document}

\title{A modified Newton method for multilinear PageRank
\thanks{.} }

\author{Pei-Chang Guo \thanks{Corresponding author, e-mail: guopeichang@pku.edu.cn} \\
School of Science, China University of Geosciences, Beijing, 100083, China}

\date{}

\maketitle
\begin{abstract}
When studying the multilinear PageRank problem, a system of polynomial equations needs to be solved. In this paper, we develop convergence theory for a modified Newton method in a particular parameter regime. The sequence of vectors produced by Newton-like method is monotonically increasing and converges to the nonnegative solution. Numerical results illustrate the effectiveness of this procedure.

\vspace{2mm} \noindent \textbf{Keywords}: multilinear PageRank, tensor, Newton-like method, monotone convergence.
\end{abstract}

\section{Introduction}
When receiving a search query, Google' search engine could find an immense set of web pages that contained virtually the same words  as the user entered.  To determine the importance of web pages, Google devised a system of scores called PageRank \cite{page}.
A random web surfer, with probability $\alpha$ randomly transitions according to a column  stochastic matrix $P$, which represents the link structure of the web, and  with probability $1-\alpha$ randomly transitions according to the fixed distribution, a column stochastic vector $v$ \cite{gleich}.

The PageRank vetor $x$, which is  the stationary distribution of the PageRank Markov chain, is unique and solves the linear system
\begin{equation*}
    x=\alpha Px+(1-\alpha)v.
\end{equation*}

Gleich et al. extend PageRank to higher-order Markov chains and proposed multiliner PageRank \cite{multi}. The limiting
probability distribution vector of a transition probability tensor  discussed in \cite{li} can be seen as as a special case of multiliner PageRank.
We recall that an mth-order Markov chain $S$ is a stochastic process that satisfies
\begin{equation*}
    Pr(S_t=i_1|S_{t-1}=i_2,\cdots, S_1=i_t)=Pr(S_t=i_1|S_{t-1}=i_2,\cdots, S_{t-m}=i_{m+1}),
\end{equation*}
where the future state only relies on the past $m$ states. For a second-order n-state Markov chain $S$, its transition probabilities are $\underline{P}_{ijk}=Pr(S_{t+1}=i|S_t=j,S_{t-1}=k)$. Through modelling a random surfer on a higher-order chain, Higher-order PageRank is introduced. With probability $\alpha$, the surfer transitions according to the higher-order chain, and with probability $1-\alpha$, the surfer teleportss according to the distribution  $v$.

Let $\underline{P}$ be an order-$m$ tensor representing an $(m-1)$th order Markov chain, $\alpha$ be a probability less than 1, and $v$ be a stochastic vector. Then the multilinear PageRank vector is a nonnegative, stochastic solution of the following polynomial system:
\begin{equation}\label{muleq0}
    x=\alpha\underline{P}x^{(m-1)}+(1-\alpha) v.
\end{equation}
Here $\underline{P}x^{(m-1)}$ for a vector $x\in\mathbb{R}^n$ denotes a vector in $\mathbb{R}^n$, whose $i$th component is
\begin{equation*}
    \sum_{i_2,\ldots,i_m=1}^{n}\underline{P}_{ii_2\ldots i_m}x_{i_2}\cdots x_{i_m}.
\end{equation*}

Gleich et al. proved that when $\alpha<\frac{1}{m}$,  the multilinear PageRank equation \eqref{muleq0} has a unique solution. Five different methods are studied to compute the multilinear PageRank vector \cite{multi}. Among the five methods, Newton iteration performs well on some tough problems. It's proved that, for a third-order tensor, Newton iteration converges quadratically  when $\alpha < 1/2$ \cite{multi}.

In this paper, we give a modified Newton method for solving the multilinear PageRank vector. We show that, for a third-order tensor when $\alpha < 1/2$, starting with a suitable initial guess, the sequence of the iterative vectors generated by the modified Newton method is monotonically increasing and converges to  solution of equation \eqref{muleq0}. Numerical experiments show that the modified Newton method can be more efficient than Newton iteration.

We introduce some necessary notation for the paper. For any matrices $B=[b_{ij}]\in \mathbb{R}^{n\times n}$, we write $B\geq 0 (B>0)$ if $b_{ij} \geq 0 (b_{ij}> 0)$ holds for all $i,j$. For any matrices $A,B \in \mathbb{R}^{n\times n}$, we write $A\geq B (A > B)$ if $a_{ij}\geq b_{ij}(a_{ij} > b_{ij})$ for all $i, j$. For any vectors $x,y \in \mathbb{R}^n$ ,we write $x\geq y (x>y)$ if $x_i\geq y_i (x_i>y_i)$ holds for all $i=1, \cdots,n$. The vector of all ones is denoted by e, i.e., $e=(1, 1, \cdots, 1)^T$. The identity matrix is denoted by $I$.

The rest of the paper is organized as follows. In section 2 we recall Newton's method and present a modified Newton iterative procedure. In section 3 we prove the monotone convergence result for the modified Newton method. In section 4 we present some
numerical examples, which show that our new algorithm can be faster
than Newton method. In section 5, we give our conclusions.

\section{A Modified Newton Method}
Let $\underline{P}$ be  a third-order stochastic tensor. Let $R$ be the $n$-by-$n^2$ flattening of  $\underline{P}$ along the first index (see \cite{golub} for more on flattening of a tensor):
$$
R=\left[
        \begin{array}{cccccccccc}
         \underline{P}_{111} & \cdots & \underline{P}_{1n1}&\underline{P}_{112}&\cdots &\underline{P}_{1n2}&\cdots & \underline{P}_{11n}&\cdots & \underline{P}_{1nn}\\
         \underline{P}_{211} & \cdots & \underline{P}_{2n1}&\underline{P}_{212}&\cdots &\underline{P}_{2n2}&\cdots & \underline{P}_{21n}&\cdots & \underline{P}_{2nn}\\
         \vdots &\ddots & \vdots & \vdots& \ddots &\vdots & \vdots& \vdots & \ddots & \vdots\\
          \underline{P}_{n11} & \cdots & \underline{P}_{nn1}&\underline{P}_{n12}&\cdots &\underline{P}_{nn2}&\cdots & \underline{P}_{n1n}&\cdots & \underline{P}_{nnn}\\
        \end{array}
      \right].
$$
Here $R$ is with column sums equal to $1$.
Then \eqref{muleq0} is
\begin{equation}\label{muleq}
    \mathcal{F}(x)=x-\alpha R(x\otimes x)-(1-\alpha) v=0.
\end{equation}
The function $\mathcal{F}$ is a mapping from $\mathbb{R}^n$ into
itself and the Fr\'{e}chet derivative of $\mathcal{F}$ at $x$ is a linear map $\mathcal{F}^{'}_x: \mathbb{R}^n\rightarrow\mathbb{R}^n$ given by
\begin{equation}\label{daoshu}
   \mathcal{F}^{'}_x: z\mapsto [I-\alpha R(x\otimes I+I\otimes x)]z=z-\alpha R(x\otimes z+z\otimes x).
\end{equation}
To suppress the technical details, later we will equivalently consider $ \mathcal{F}^{'}_x$ as the matrix $[I-\alpha R(x\otimes I+I\otimes x)].$
The second derivative of $\mathcal{F}$ at $x$, $\mathcal{F}^{"}_x: \mathbb{R}^n \times \mathbb{R}^n \rightarrow\mathbb{R}^n$, is given by
\begin{equation}\label{erdaoshu}
   \mathcal{F}^{''}_x(z_1,z_2)=-\alpha R(z_1\otimes z_2+z_2\otimes z_1).
\end{equation}

For a given $x_0$, the Newton sequence for the solution of $\mathcal{F}(x) = 0$ is
\begin{eqnarray}\label{newit}
 \nonumber x_{k+1}&=&x_k-(\mathcal{F}^{'}_{x_k})^{-1}\mathcal{F}({x_k})  \\
   &=& x_k-[I-\alpha R(x_k\otimes I+I\otimes x_k)]^{-1}\mathcal{F}(x_k),
\end{eqnarray}
for $k=0,1,\cdots,$ provided that $\mathcal{F}^{'}_{x_k}$ is invertible for all $k$.

As we see, for the nonlinear equation $\mathcal{F}(x)=0$, the sequence generated by Newton iteration will converge quadratically to the the solution \cite{gleich}. However, there is a disadvantage with Newton method. At every Newton iteration step,we need to compute the Fr\'{e}chet derivative and perform an LU factorization. In order to save the overall cost, we present the  modified Newton algorithm for equation \eqref{muleq} as follows.

\vspace{3mm}

\textbf{ Modified Newton Algorithm for Equation \eqref{muleq}}\\

Given initial value $x_{0,0}$, for $i=0,1,\cdots$
\begin{eqnarray}
\nonumber  x_{i,s} &=& x_{i,s-1}-(\mathcal{F}^{'}_{x_{i,0}})^{-1}\mathcal{F}({x_{i,s-1}}),   \quad s=1,2,\cdots , n_i, \\
\label{likea} &=& x_{i,s-1}-(I-\alpha R(x_{i,0} \otimes I+I\otimes x_{i,0}))^{-1}\mathcal{F}({x_{i,s-1}}), \quad s=1,2,\cdots , n_i, \\
 \label{likeb} x_{i+1,0} &= &x_{i,n_i}
\end{eqnarray}

From equation (\ref{likea}) we see that Newton method results when $n_i=1$ for $i=0,1,\cdots,$ and the chord method \cite{kelley} results when $n_0=\infty$. The chord method needs totally one LU factorization so the cost for each iteration step is low. But the convergence rate of the chord method is very slow.

\section{Convergence Analysis}
In this section, we prove a monotone convergence result for the  modified Newton method for equation \eqref{muleq}.
\subsection{preliminary}
We first recall that a real square matrix $A$ is called a Z-matrix if
all its off-diagonal elements are nonpositive. Note that any Z-matrix A can be
written as $sI-B$ with $B \geq 0$. A Z-matrix $A$ is called an M-matrix if $s\geq \rho(B)$,
where $\rho(\cdot)$ is the spectral radius; it is a singular M-matrix if $s=\rho(B)$ and a
nonsingular M-matrix if $s>\rho(B)$.
We will make use of the following result (see \cite{varga}).
\begin{lemma}\label{yubei1}
For a Z-matrix $A$, the following are equivalent:
\begin{itemize}
  \item [$(a)$] $A$ is a nonsingular M-matrix.
  \item [$(b)$] $A^{-1}\geq 0$ .
  \item [$(c)$] $Av>0$ for some vector $v>0$.
  \item [$(d)$] All eigenvalues of $A$ have positive real parts.
\end{itemize}
\end{lemma}
The next result is also well known and also can be found in \cite{varga}.
\begin{lemma}\label{yubei2}
Let $A$ be a nonsingular M-matrix. If $B \geq A$ is a Z-matrix, then $B$ is also  nonsingular M-matrix . Moreover, $B^{-1}\leq A^{-1}$.
\end{lemma}

\subsection{Monotone convergence}
The next lemma displays the monotone convergence properties of Newton iteration for  \eqref{muleq}.  \begin{lemma}\label{lemm1}
Suppose that a vector $x$ is such that
\begin{itemize}
  \item [(i)] $\mathcal{F}(x)\leq 0$,
  \item [(ii)] $0\leq x $, and  $e^Tx \leq 1$.
\end{itemize}
Then there exists the vector
\begin{equation}\label{zheng1}
    y=x-(\mathcal{F}^{'}_{x})^{-1}\mathcal{F}(x)
\end{equation}
such that
\begin{itemize}
  \item [(a)] $\mathcal{F}(y)\leq 0$,
  \item [(b)] $0\leq x\leq y$, and $e^Ty \leq 1$.
\end{itemize}
\end{lemma}
\begin{proof}
Note that $R$ is with all  column sums equal to $1$, so both $R(x\otimes I)$ and $R(I\otimes x)$ are nonnegative matrices whose column sums are $e^Tx$.
If  $0\leq x $ and $e^Tx \leq 1$, then $\mathcal{F}^{'}_{x}=I-\alpha R(x\otimes I+I\otimes x)$ is strictly diagonally dominant and thus a nonsingular M-matrix. So from lemma \ref{yubei1} and condition $(i)$, $y$ is well defined and $y-x\geq 0$.

From equation \eqref{zheng1} and Taylor formula,
we have
   \begin{eqnarray*}\mathcal{F}(y)&=&\mathcal{F}(x)+\mathcal{F}^{'}_{x}(y-x)+\frac{1}{2}\mathcal{F}^{''}_x(y-x,y-x)\\
       \nonumber &=&\frac{1}{2}\mathcal{F}^{''}_x(y-x,y-x)  \\
      \nonumber &=&-\alpha R[(y-x)\otimes (y-x)] \leq 0
     \end{eqnarray*}
We now prove the second term of (b). A mathematically equivalent form of \eqref{zheng1}
is
\begin{equation}\label{pzhongjian}
   [I-\alpha R(x\otimes I+I\otimes x)](y-x)=\alpha R(x\otimes x)+(1-\alpha) v-x.
   \end{equation}
  Taking summations  on both sides of equation \eqref{pzhongjian},
 we get
 \begin{equation*}
   [1-2\alpha (e^Tx)](e^Ty-e^Tx)=\alpha (e^Tx)^2+(1-\alpha)-(e^Tx),
 \end{equation*}
 which is equivalent to
 \begin{equation}\label{pzhongjian2}
   e^Ty=\frac{1-\alpha-\alpha(e^Tx)^2}{1-2\alpha(e^T x)}.
   \end{equation}

Combining \eqref{pzhongjian2} and $e^Tx \leq 1$, we know $e^Ty>1$ doesn't hold, thus $e^Ty \leq 1$.
\end{proof}
The next lemma is an extension of Lemma \ref{lemm1}, which will be the theoretical basis of monotone convergence result of Newton-like method for \eqref{muleq}.

\begin{lemma}\label{lemm2}
Suppose there is a vector $x$ such that
\begin{itemize}
  \item [(i)] $\mathcal{F}(x)\leq 0$,
  \item [(ii)] $0\leq x$, and $e^T x\leq 1$.
\end{itemize}
Then for any vector $z$ with $0\leq z\leq x$, there exists the vector
\begin{equation}\label{qzheng2}
    y=x-(\mathcal{F}^{'}_{z})^{-1}\mathcal{F}(x)
\end{equation}
such that
\begin{itemize}
  \item [(a)] $\mathcal{F}(y)\leq 0$,
  \item [(b)] $0\leq x\leq y$, and $e^Ty \leq 1$.
\end{itemize}
\end{lemma}
\begin{proof}
Here let
$$ \hat{y}=x-(\mathcal{F}^{'}_{x})^{-1}\mathcal{F}(x).$$
First, from Lemma \ref{lemm1}, we know that $\mathcal{F}^{'}_x$ is a nonsingular M-matrix.
 Because $0\leq z\leq x$ and Lemma \ref{yubei2}, we know that $\mathcal{F}^{'}_z$ is also a nonsingular M-matrix and $$ 0\leq [\mathcal{F}^{'}_z]^{-1}\leq [\mathcal{F}^{'}_x]^{-1}.$$
 So the vector $y$ is well defined and $0\leq x \leq y \leq \hat{y} $.
 From  Lemma \ref{lemm1},  we know $e^T\hat{y}\leq 1$, so $e^Ty\leq 1$. So (b) is true.
We have
  \begin{eqnarray*} \mathcal{F}(y)&=&\mathcal{F}(x)+\mathcal{F}^{'}_x(y-x)+\frac{1}{2}\mathcal{F}^{''}_x(y-x,y-x)\\
       \nonumber &=&\mathcal{F}(x)+\mathcal{F}^{'}_z(y-x)+(\mathcal{F}^{'}_x-\mathcal{F}^{'}_z)(y-x)+\frac{1}{2}\mathcal{F}^{''}_x(y-x,y-x)  \\
       \nonumber &=&\mathcal{F}^{''}_x(x-z,y-x)+\frac{1}{2}\mathcal{F}^{''}_x(y-x,y-x)  \\
      \nonumber & \leq& 0,
     \end{eqnarray*}
 the last inequality holds because $x-z\geq 0$ and $y-x\geq 0$. So (a) is true.
\end{proof}

Using Lemma \ref{lemm2}, we can get the following monotone convergence result of Newton-like methods for \eqref{muleq}. For  $i=0,1,\cdots$, we will use $x_i$ to denote $x_{i,0}$ in the Newton-like algorithm \eqref{likeb}, thus $x_i=x_{i,0}= x_{i-1,n_{i-1}}$.
\begin{theorem}\label{thm1}
Suppose that a vector $x_{0,0}$ is such that
\begin{itemize}
  \item [(i)] $\mathcal{F}(x_{0,0})\leq 0$,
  \item [(ii)] $0\leq x_{0,0}$, and $e^Tx_{0,0}\leq 1$.
\end{itemize}
Then the Newton-like algorithm \eqref{likea},\eqref{likeb} generates a sequence $\{x_{k}\}$
such that $x_k \leq x_{k+1} $ for all $k \geq 0$, and $ \lim_{k \to \infty} \mathcal{F}(x_k)=0$.
\end{theorem}
\begin{proof}
We prove the theorem by mathematical induction. From Lemma \ref{lemm2}, we have
\begin{equation*}
     x_{0,0}\leq \cdots \leq x_{0,n_0}=x_1,
\end{equation*}
\begin{equation*}
    \mathcal{F}(x_1)\leq 0,
\end{equation*}
and
\begin{equation*}
    e^Tx_1\leq 1.
\end{equation*}
Assume $ e^Tx_i\leq 1$,
\begin{equation*}
    \mathcal{F}(x_i)\leq 0,
\end{equation*}
and
\begin{equation*}
    x_{0,0}\leq \cdots \leq x_{0,n_0}=x_1 \leq \cdots \leq x_{i-1,n_{i-1}}=x_{i}.
\end{equation*}
Again by Lemma \ref{lemm2} we have
\begin{equation*}
    \mathcal{F}(x_{i+1})\leq 0,
\end{equation*}
\begin{equation*}
    x_{i,0}\leq \cdots \leq x_{i,n_i}=x_{i+1},
\end{equation*}
and $e^Tx_{i+1}\leq 1$.
Therefore we have proved inductively the sequence $\{x_k\}$ is monotonically increasing and bounded above. So it has a limit $x_*$. Next we show that $\mathcal{F}(x_*)=0$.
Since $x_0\leq x_k$, from Lemma \ref{yubei2} we have
$$0\leq (\mathcal{F}^{'}_{x_0})^{-1} \leq (\mathcal{F}^{'}_{x_k})^{-1} .$$

Let $i\rightarrow \infty$ in $x_{i+1}\geq x_{i,1}=x_i-(\mathcal{F}^{'}_{x_i})^{-1}\mathcal{F}({x_i})\geq x_i-(\mathcal{F}^{'}_{x_0})^{-1}\mathcal{F}({x_i})\geq 0$,
then we get $$\lim_{i\rightarrow \infty}(\mathcal{F}^{'}_{x_0})^{-1}\mathcal{F}({x_i})=0.$$
$\mathcal{F}(x)$ is continuous at $x_*$, so $(\mathcal{F}^{'}_{x_0})^{-1}\mathcal{F}({x_*})=0,$ then we get
$\mathcal{F}({x_*})=0.$

\end{proof}

\section{Numerical Experiments}
We remark that the modified Newton method differs from Newton's method in that the evaluation and factorization of the  Fr\'{e}chet derivative are not done at every iteration step. So, while more iterations will be needed than for Newton's method, the overall cost of the modified Newton method could be much less. Our numerical experiments confirm the efficiency of the modified Newton method for equation \eqref{muleq}.

About how to choose the optimal scalars $n_i$ in the Newton-like  algorithm (\ref{likea}), now we have no theoretical results. This is a goal for our future research. In our extensive numerical experiments, we update the Fr\'{e}chet derivative every four iteration steps. That is, for $i=0,1,\cdots$  we choose $n_i=4$ in the Newton-like  algorithm (\ref{likea}).

We define the number of the factorization of the Fr\'{e}chet derivative in the algorithm as the outer iteration steps, which is $i+1$ when $s>0$ or $i$ when $s=0$ for an approximate solution $x_{i,s}$ in the modified Newton algorithm.

The outer iteration steps (denoted as ``iter"),  the
elapsed CPU time in seconds (denoted as ``time"), and the normalized residual
(denoted as "NRes" ) are used to measure the
feasibility and effectiveness of our new method, where "NRes" is
defined as
\begin{equation*}
\mbox{NRes}=\frac{\parallel \tilde{x}-\alpha R(\tilde{x}\otimes \tilde{x})-(1-\alpha) v\parallel_1}{(1-\alpha)\parallel  v\parallel_1 + \alpha\parallel R(\tilde{x}\otimes \tilde{x})\parallel_1+\parallel \tilde{x}\parallel_1 },
\end{equation*}
where $\parallel\cdot\parallel_1$ denotes the 1-norm of the vector and $\tilde{x}$ is an approximate solution to the solution of \eqref{muleq}. We use $x=0$ as the initial iteration value of the Newton-like method. The numerical  tests were performed on a laptop (2.4 Ghz and 2G Memory) with MATLAB R2013b. Numerical experiments show that the the modified Newton method could be more efficient than Newton iteration. We present the numerical results for a random-generated problem.  The MATLAB code used for its generation is reported here. The problem size is $n=300$ in Table 1. %Notice that the matrix $R\in $\mathbb{R}^{n\timesn^2}$.

\begin{quote}
function[R,v]=page(n)\\
v=ones(n,1);\\
N=n*n;\\
rand('state',0);\\
R=rand(n,n*n);\\
s=v'*R;\\
for i=1:N\\
    R(:,i)=R(:,i)/s(i);\\
end\\
v=v/n;\\
\end{quote}

\begin{table}[h]
\label{tab1}
\begin{center}
\caption{Comparison of the numerical results}
\begin{tabular}{|c|c|c|c|c|}\hline
$\alpha$ & Method &  time & NRes & iter\\
\hline  & Newton&  24.648 &  5.19e-13& 9 \\
\cline{2-5} 0.490 & modified Newton & 18.580 & 8.79e-13 & 5 \\
\hline  & Newton&  28.782 &  1.29e-13& 10 \\
\cline{2-5} 0.495 & modified Newton & 19.656 & 3.11e-12 & 5 \\
\hline  & Newton&  32.745 &  3.07e-13& 12 \\
\cline{2-5} 0.499 & modified Newton & 23.275& 9.63e-12 & 6 \\

\hline
\end{tabular}
\end{center}
\end{table}

\section{Conclusions}
In this paper, the application of the Newton-like method to the polynomial system of equations   arising from the multilinear PageRank problem has been considered. The convergence analysis shows that this method is feasible in a particular parameter regime. Numerical calculations show that the modified Newton method can outperform Newton's method.

\end{document}